\documentclass[a4paper,10pt]{amsart}
\usepackage[czech,english]{babel}
\usepackage{amsmath,amssymb,amsthm,amscd}
\usepackage{graphicx}
\usepackage[all]{xy}

\newcommand{\ModR}{\hbox{{\rm Mod-}}R}

\newcommand{\Add}{\mathrm{Add}}

\newcommand{\C}{\mathcal{C}}
\newcommand{\A}{\mathcal{A}}

\newcommand{\B}{\mathcal{B}}

\DeclareMathOperator{\Hom}{Hom}

\DeclareMathOperator{\End}{End}
\DeclareMathOperator{\Ext}{Ext}

\DeclareMathOperator{\Ker}{Ker}
\DeclareMathOperator{\Img}{Im}

\DeclareMathOperator{\Filt}{Filt}

\DeclareMathOperator{\ident}{id}

\usepackage[usenames]{color}

\theoremstyle{plain}
\newtheorem{thm}{Theorem}[section]
\newtheorem{prop}[thm]{Proposition}
\newtheorem{lem}[thm]{Lemma}

\newtheorem{cor}[thm]{Corollary}
\theoremstyle{definition}
\newtheorem{defn}[thm]{Definition}
\newtheorem{exm}[thm]{Example}

\theoremstyle{remark}
\newtheorem*{rema}{Remark}

\begin{document}
\title{Enochs conjecture for cotorsion pairs and more}

\author{\textsc{Silvana Bazzoni}}
\address{Dipartimento di Matematica Tullio Levi-Civita, Universit\`a di Padova \\
Via Trieste 63, 35121 Padova, Italy}
\email{bazzoni@math.unipd.it}

\author{\textsc{Jan \v Saroch}}
\address{Department of Algebra, Faculty of Mathematics and Physics, Charles University \\ 
Sokolovsk\'{a} 83, 186 75 Praha~8, Czech Republic}
\email{saroch@karlin.mff.cuni.cz}


\thanks{Research of the second-named author supported by GA\v CR 23-05148S}


\begin{abstract} Enochs Conjecture asserts that each covering class of modules (over any ring) has to be closed under direct limits. Although various special cases of the conjecture have been verified, the conjecture remains open in its full generality. In this paper, we prove the conjecture for the classes $\Filt(\mathcal S)$ where $\mathcal S$ consists of $\aleph_n$-presented modules for some fixed $n<\omega$. In particular, this applies to the left-hand class of any cotorsion pair generated by a~class of $\aleph_n$-presented modules.

Moreover, we also show that it is consistent with ZFC that Enochs Conjecture holds for all classes of the form $\Filt(\mathcal S)$ where $\mathcal S$ is a~set of modules. This leaves us with no explicit example of a~covering class where we cannot prove that Enochs Conjecture holds (possibly under some additional set-theoretic assumption).
\end{abstract} 

\maketitle
\vspace{4ex}

\section*{Introduction}
\label{sec:intro}
Approximation theory was developed as a~tool to approximate arbitrary modules by modules in certain classes where the classification is more manageable. Left and right approximations were studied in the case of modules over finite dimensional algebras by Auslander, Reiten, and Smal{\o} and independently by Enochs and Xu for modules over arbitrary rings using the terminology of preenvelopes and precovers. An important problem in approximation theory is when minimal approximations, that is covers or envelopes, over certain classes exist (see Section~\ref{sec:prelim} for definitions).

Bass in~\cite{Bass} studied the existence of minimal approximations for the class of projective modules and he showed that every module has a~projective cover if and only if the class of projective modules is closed under direct limits. 

A theorem of Enochs says that if a~class~$\C$ in $\ModR$ is closed under direct limits, then any module that has a~$\C$-precover has a~$\C$-cover, cf.\ \cite{Eno}. The converse problem, that is if a~covering class $\mathcal{C}$ is necessarily closed under direct limits, is still open and is known as Enochs Conjecture.

Some significant advancements have been made towards the solution of Enochs Conjecture in recent years. In 2017, Angeleri H\"ugel--\v{S}aroch--Trlifaj in \cite{AST17} proved the validity of Enochs Conjecture for the left hand class $\A$ of a~cotorsion pair $(\mathcal{A}, \mathcal{B})$ such that $\mathcal{B}$ is closed under direct limits. In particular, this holds for all tilting cotorsion pairs. The result in \cite{AST17} is based on set-theoretical methods developed by the second-named author in \cite{Sa0}.

Moreover, Bazzoni--Positselski--\v S\v tov\'i\v cek in \cite{BPS0} proved, using a~purely algebraic approach, that if a~cotorsion pair $(\A, \B)$ is given such that $\B$ is closed under direct limits and $\Add (M)$ is covering for a~module $M \in \A \cap \B$, then $\Add (M)$ is closed under direct limits. In particular, this applies to tilting cotorsion pairs.

In \cite{BL0}, Enochs Conjecture has been verified for the class $\mathcal P_1$ of modules of projective dimension at most 1 over a commutative semihereditary ring.

Even more recently, in \cite{Sa2}, the second-named author has proved the validity of Enochs Conjecture for the class $\Add (M)$ where $M$ decomposes into a~direct sum of modules at most $\aleph_n$-generated, for some fixed $n<\omega$, while, adding a~set-theoretic assumption (similar in flavour to the one used in Section~\ref{sec:cons}), the conjecture holds for $\Add (M)$ where $M$ is arbitrary.

In the present paper, we show (Theorem~\ref{t:pure-epi}, Corollary~\ref{c:Silvana}) that, for any infinite regular cardinal $\kappa$, every $<\!\kappa$-Kaplansky class of modules is closed under taking $\kappa$-pure-epimorphic images (hence, in particular, under $\kappa$-directed limits) provided that it is covering. In particular, this applies to the left-hand class of any cotorsion pair generated by $<\!\kappa$-presented modules (Corollary~\ref{c:Silvana-cotpairs}).

The next step is to use the powerful tool given by the construction of a~``tree module'' which goes back to \cite{ES} and which was already successfully applied in many situations like in~\cite{BS}, \cite{Sa1} and \cite{Sa2}. A suitable construction of a~``tree module'' and a~cardinal counting argument prove the main result of Section~2 (Corollary~\ref{c:aleph_n}), namely that every class of the form $\Filt(\mathcal S)$ where $\mathcal S$ is a~set of $<\!\aleph_n$-presented modules, for some $n<\omega$, satisfies Enochs Conjecture.
 
In Section~3 (Corollary~\ref{c:conEnochs}), we prove that, under a~suitable set-theoretic assumption, Enochs Conjecture holds for all the classes of the form $\Filt(\mathcal S)$ where $\mathcal S$ is a~set of modules.

\section{Preliminaries}
\label{sec:prelim}

Throughout this paper, $R$ denotes an associative unital ring and $\ModR$ the category of all (right $R$-)modules and homomorphisms between them. We note, however, that our results generalise in a straightforward way to the category of unitary modules over a ring with enough idempotents.

All classes $\mathcal A\subseteq\ModR$ are assumed to be closed under isomorphic images.

Unless otherwise stated, we work in the classical realm of ZFC, i.e.\ Zermelo--Fraenkel set theory with the axiom of choice. We also use the usual definitions of ordinal numbers, cardinal numbers, cofinality and stationary subsets (in particular, $\omega$ denotes the least infinite ordinal number). The reader unfamiliar with some of these notions is encouraged to consult the monograph \cite{J0} or \cite{EM0}.

Given an infinite (regular) cardinal $\kappa$, we say that a~partially ordered set $ (I , \leq)$ is \emph{$\kappa$-directed} provided that each subset $J\subseteq I$ with $|J| < \kappa$ has an upper bound in $ (I , \leq)$. A system $\mathcal M = (M_i, f_{ji}\colon M_i \to M_j \mid i\leq j\in I)$ of modules and homomorphisms between them indexed by a $\kappa$-directed poset $(I, \leq)$ is called \emph{$\kappa$-directed} as well.

Let $\kappa$ be an infinite (regular) cardinal and $(I,\leq)$ an upward directed partially ordered set. A directed system $\mathcal M = (M_i,f_{ji}\colon M_i\to M_j \mid i\leq j\in I)$ of modules and homomorphisms between them is called \emph{$\kappa$-continuous} provided that, for each \emph{totally ordered} $J\subseteq I$ of cardinality $<\kappa$, there exists $\sup J\in I$ and $M_{\sup J} = \varinjlim _{j\in J} M_j$. It is easy to show that, in this case, $\sup J$ also exists in $(I,\leq)$ for each \emph{directed subset} $J$ of $I$ with $|J|<\kappa$, and $M_{\sup J} = \varinjlim _{j\in J} M_j$ holds as well. In particular, $\mathcal M$ is $\kappa$-directed.

An epimorphism $\pi\colon B\to C$ is said to be \emph{$\kappa$-pure} if $\Hom_R(F,\pi)$ is surjective whenever $F$ is a $<\!\kappa$-presented module. Correlatively, a~monomorphism $\nu\colon A\to B$ is \emph{$\kappa$-pure} if the canonical projection $B\to B/\Img(\nu)$ is $\kappa$-pure. If $\kappa = \aleph_0$, it is often omitted. For a~$\kappa$-continuous system $\mathcal M$ as above (in fact, for any $\kappa$-directed system $\mathcal M$), the canonical short exact sequence $0\to K\to \bigoplus_{i\in I}M_i \to \varinjlim\mathcal M \to 0$ is $\kappa$-pure, i.e.\ comprises of $\kappa$-pure monomorphism and $\kappa$-pure epimorphism.

\smallskip

For a class $\C\subseteq \ModR$, we put \[\mathcal C^\perp =\{M \in \ModR \mid (\forall C\in\mathcal C)\,\Ext_R^1(C,M) = 0\}\] and \[{}^\perp \mathcal C = \{M\in\ModR \mid (\forall C\in\mathcal C)\,\Ext_R^1(M,C) = 0\}.\]

We recall the notion of a~precover, or sometimes called right approximation. If~$\C$ is any class of modules and $M\in \ModR$, a homomorphism $f\colon C \to M$ is called a~$\C$-\emph{precover} of $M$, if $C \in \C$ and $\Hom_R(C^\prime, f)$ is surjective for each $C^\prime \in\C$.

A $\C$-precover $\varphi\in \Hom_R(C, M)$ is called a~$\C$-\emph{cover} (or a minimal right approximation) of $M$ if for every endomorphism $f$ of $C$ such that $\varphi=\varphi f$, $f$ is an automorphism of $C$. So a~$\C$-cover is a~minimal version of a~$\C$-precover.

A $\C$-precover $\varphi$ of $M$ is said to be \emph{special} if $\varphi$ is an epimorphism and $\Ker \varphi\in \C^{\perp}$. The class $\C$ is called \emph{(special) precovering}  (\emph{covering}, resp.) if each $M \in \ModR$ admits a~(special) $\C$-precover (cover, resp.). Note that any covering class of modules is closed under direct summands and direct sums.

\smallskip

Recall that a \emph{filtration} of a module $M$ is a sequence $\mathfrak F = (M_\alpha\mid \alpha\leq \sigma)$ of modules indexed by an ordinal $\sigma$ such that $M_0 = 0$, $M_\sigma = M$, $M_\alpha\subseteq M_\beta$ for each $\alpha<\beta\leq\sigma$ and $M_\alpha = \bigcup_{\beta<\alpha} M_\beta$ for each limit ordinal $\alpha\leq\sigma$. Moreover, if the consecutive factors in $\mathfrak F$ are isomorphic to modules from a~given class $\mathcal A\subseteq\ModR$, we say that $\mathfrak F$ is an~\emph{$\mathcal A$-filtration} of $M$ and, correlatively, that $M$ is \emph{$\mathcal A$-filtered}.

Given a~class $\mathcal A\subseteq\ModR$, we denote by $\Filt(\mathcal A)$ the class of all modules possessing an~$\mathcal A$-filtration; otherwise put, the class of all $\mathcal A$-filtered modules. Furthermore, given a~cardinal $\kappa$, we denote by $\mathcal A^{<\kappa}$ the subclass of $\mathcal A$ consisting of all the $<\!\kappa$-presented modules in $\mathcal A$. It follows from \cite[Theorem~7.21]{GT0} that $\Filt(\mathcal A^{<\kappa})$ is a~precovering class of modules for any $\mathcal A$ and~$\kappa$.

Given an infinite regular cardinal $\kappa$, we say that a~class $\mathcal A$ of modules is \emph{$<\!\kappa$-Kaplansky} provided that for each $A\in\mathcal A$ and $X\subseteq A$ with $|X|<\kappa$, there exists a~$<\!\kappa$-presented module $B\in\mathcal A$ such that $X\subseteq B\subseteq A$ and $A/B\in\mathcal A$. If $\kappa = \lambda^+$, we also write $\lambda$-Kaplansky instead of $<\!\kappa$-Kaplansky. This agrees with \cite[Definition~10.1]{GT0}.

We say that a~pair $\mathfrak C = (\mathcal A,\mathcal B)$ of classes of modules is a~\emph{cotorsion pair} if $\mathcal A^\perp = \mathcal B$ and ${}^\perp\mathcal B = \mathcal A$. Moreover, for a~class of modules $\mathcal S$ such that $\mathcal B = \mathcal S^\perp$, we say that $\mathfrak C$ is \emph{generated by $\mathcal S$}. Recall that by \cite[Theorem~6.11]{GT0} (essentially, by the small object argument), the class $\mathcal A$ in a~cotorsion pair $(\mathcal A,\mathcal B)$ generated by a~set (or, equivalently, by a~skeletally small class) is special precovering, and $\mathcal A = \Filt(\mathcal A^{<\kappa})$ for a~suitable (regular) cardinal $\kappa$, cf.\ \cite[Theorem~7.13]{GT0}.

\section{The main theorem for $\Filt(\mathcal S)$ and cotorsion pairs}
\label{sec:main}

We start with a very useful result which allows us to capitalize on the covering assumption. It comes from \cite{BPS0}. Recall that a~monomorphism $m\colon B\to A$ is called \emph{locally split} provided that \[(\forall x\in B)(\exists h\in\Hom_R(A,B))\, h(m(x))=x.\]

\begin{prop} \label{p:localsplit} Let $\mathcal A\subseteq \ModR$ and $f\in\Hom_R(A, M)$ be a~surjective $\mathcal A$-precover with a~locally split kernel. If $M$ has an $\mathcal A$-cover, then $M\in\mathcal A$ and the epimorphism $f$ splits.
\end{prop}

\begin{proof} In \cite[Corollary~5.3]{BPS0}, a~general version for Ab5 categories is proved. For the reader's convenience, we include here a~short module-theoretic proof.

Since $M$ has an $\mathcal A$-cover, there is a~decomposition $A = C \oplus D$ such that $D\subseteq \Ker(f)$ and $f\restriction C\colon C\to M$ is an $\mathcal A$-cover (see, e.g., \cite[Lemma~5.8]{GT0}). Put $g = f\restriction C$ and $B = \Ker(g)$. Then $B \oplus D = \Ker(f)$, whence $g$ has a~locally split kernel, too.

Let $b\in B$ be arbitrary and let $h\in\Hom_R(C,B)$ be such that $h(b) = b$. Then $\ident_C-h\in\End_R(C)$ has $b$ in its kernel. On the other hand, the fact that $g$ is a~cover and $g = g\circ (\ident_C-h)$ implies that $\ident_C-h$ is an~automorphism. It follows that $b = 0$, and thus $B = 0$ as well. We conclude that $g$ is an~isomorphism, so $C\cong M\in\mathcal A$ and $f$ splits.
\end{proof}

Before stating our stepping-stone result of this section, recall that, for any infinite regular $\kappa$ and $\mathcal A\subseteq \ModR$, the class $\Filt(\mathcal A^{<\kappa})$ is $<\!\kappa$-Kaplansky by \cite[Theorem~10.3(a)]{GT0}\footnote{Replace $\kappa^+$ by $\kappa$, $\leq\kappa$ by $<\!\kappa$ and $\kappa$-Kaplansky by $<\!\kappa$-Kaplansky in its statement and proof.}. Similarly, by the Walker's lemma, the class $\Add(N)$ of all direct summands of direct sums of copies of $N$ is $\lambda$-Kaplansky whenever $N$ decomposes as a~direct sum of $\lambda$-presented modules.

\begin{thm} \label{t:pure-epi} Let $\kappa$ be an infinite regular cardinal and $\mathcal A\subseteq \ModR$ be a~$<\!\kappa$-Kaplansky class; e.g.\ let $\mathcal A = \Filt(\mathcal A^{<\kappa})$. Let $M$ be a~$\kappa$-pure-epimorphic image of a~module from $\mathcal A$. Then every special $\mathcal A$-precover $f\colon A\to M$ is a~$\kappa$-pure epimorphism with locally split kernel. Moreover, if $M$ possesses an $\mathcal A$-cover, then $M\in\mathcal A$ and $f$ splits.
\end{thm}

\begin{proof} Let $f^\prime\colon A^\prime\to M$ be a $\kappa$-pure epimorphism with $A^\prime\in\mathcal A$. Since $f^\prime$ factorizes through any $\mathcal A$-precover of $M$, every (special) $\mathcal A$-precover of $M$ has to be a~$\kappa$-pure epimorphism as well. Let us fix an arbitrary special $\mathcal A$-precover $f\colon A\to M$. We have a~short exact sequence \[0\longrightarrow B \overset{\subseteq}{\longrightarrow} A \overset{f}\longrightarrow M \longrightarrow 0\eqno{(\triangle)}\] with $B\in\mathcal A^\perp$.

Let $x\in B$ be arbitrary. Since $\mathcal A$ is $<\!\kappa$-Kaplansky, we get a~$<\!\kappa$-presented submodule $P$ of $A$ with $x\in P$ and $A/P\in\mathcal A$. Then $P/xR$ is $<\!\kappa$-presented as well.

Since $xR\subseteq B$, the epimorphism $f$ factorizes as $g\pi$ where $\pi\colon A\to A/xR$ is the canonical projection and $g\colon A/xR \to M$ is an epimorphism. Put $h = g\restriction P/xR$. We have the following commutative diagram with exact rows

\[\xymatrix{0 \ar[r] & xR \ar[r]^-{\subseteq} \ar[d]_-{i} & P \ar[r]^-{\pi\restriction P} \ar@{-->}[dl]_-{j} \ar[d]^-{\subseteq} & P/xR \ar[r] \ar[d]^{h} \ar@{-->}[dl]^-{e} & 0 \\
0 \ar[r] & B \ar[r]_-{\subseteq} & A \ar[r]_-{f} & M \ar[r] & 0
}\]
where $i$ denotes the inclusion, and there exists $e\in\Hom_R(P/xR,A)$ such that $fe = h$ since $P/xR$ is $<\!\kappa$-presented and $f$ is a~$\kappa$-pure epimorphism. Consequently, there is a~homomorphism $j\colon P\to B$ extending $i$. Now we apply the functor $\Hom_R(-, B)$ to the short exact sequence $0\to P \overset{\subseteq}{\longrightarrow} A \longrightarrow A/P\longrightarrow 0$ and use that $B\in \B$ and $A/P\in \A$ to obtain that $j$ can be further extended to $k\in\Hom_R(A,B)$. So $k(x) = i(x) = x$, and we have proved that the kernel of $f$ is locally split as $x\in B$ was arbitrary.

The moreover clause now follows directly from Proposition~\ref{p:localsplit} and from the fact that there exists a~special $\mathcal A$-precover of $M$; since, e.g., any (surjective) $\mathcal A$-cover is special by \cite[Lemma~5.13]{GT0}.

\end{proof}

Note that, instead of the cyclic submodule $xR$ of $B$, we could have considered any $<\!\kappa$-generated submodule $S$ of $B$ in the proof without changing the reasoning. Thus we obtain an~even stronger version of local splitness where $x$ is replaced by an~$X\subseteq B$ of cardinality $<\!\kappa$.

\begin{cor} \label{c:Silvana} Let $\kappa$ be an infinite regular cardinal and $\mathcal A\subseteq \ModR$ be a~covering class of modules which is also $<\!\kappa$-Kaplansky. Then $\mathcal A$ is closed under taking $\kappa$-pure-epimorphic images. In particular, if $\mathcal M$ is a~$\kappa$-directed system of modules from~$\mathcal A$, then $\varinjlim \mathcal M\in\mathcal A$.
\end{cor}

Let us also state explicitly the result for cotorsion pairs.

\begin{cor} \label{c:Silvana-cotpairs} Let $\kappa$ be an infinite regular cardinal and $(\mathcal A,\mathcal B)$ be a~cotorsion pair generated by a~class consisting of $<\!\kappa$-presented modules. If $\mathcal A$ is covering, then $\mathcal A$ is closed under taking $\kappa$-pure-epimorphic images.

In particular, the left-hand class of a~cotorsion pair generated by a~class of finitely presented modules is covering if and only if it is closed under pure-epimorphic images, if and only if it is closed under direct limits.
\end{cor}

\begin{proof} If $\kappa$ is uncountable, it follows from \cite[Theorem~7.13]{GT0} that $\mathcal A = \Filt(\mathcal A^{<\kappa})$. For $\kappa = \aleph_0$, $A\in\mathcal A$ need not be $\mathcal A^{<\kappa}$-filtered. However, we can use \cite[Corollary~6.13]{GT0} to find a~module $C\in\mathcal A\cap\mathcal B$ such that $A\oplus C$ is $\mathcal A^{<\kappa}$-filtered. Consequently, we replace the short exact sequence $(\triangle)$ in the proof of Theorem~\ref{t:pure-epi} by
\[0\longrightarrow B\oplus C \overset{\subseteq}{\longrightarrow} A \oplus C \overset{f\oplus 0}\longrightarrow M \longrightarrow 0\] and follow the proof to show that the inclusion in this modified short exact sequence is locally split. It immediately follows that the inclusion in $(\triangle)$ is locally split as well.
\end{proof}

In what follows, we prove that, if the conclusion of Corollary~\ref{c:Silvana} holds for some uncountable $\kappa<\aleph_\omega$, then $\mathcal A$ is closed under \emph{arbitrary} direct limits assuming that $\mathcal A = \Filt(\mathcal A^{<\kappa})$. This is the case, e.g., if $\mathcal A$ is the left-hand class of a~cotorsion pair generated by $<\!\kappa$-presented modules. Thus we verify the Enochs Conjecture for such classes $\mathcal A$.

To achieve our goal, we show that, for any infinite regular cardinal $\lambda$, the class $\mathcal A$ in question is closed under taking direct limits of $\lambda$-continuous directed systems of modules provided that this closure property holds for $\lambda^+$-continuous directed systems. We adapt the construction from~\cite[Lemma~A.1]{Sa2}, originally going back to \cite{ES}. First, we need to fix some notation.

\begin{defn} \label{d:treemod} Let $\kappa, \lambda$ be infinite cardinals, $\kappa$ regular. Let a~set $T\subseteq {}^\kappa \lambda$ be given where ${}^\kappa \lambda$ denotes the set of all mappings from $\kappa$ into $\lambda$. Furthermore, let $N$ be a~module and $\mathfrak K = (K_\alpha\mid \alpha\leq\kappa)$ be a~filtration of a~submodule $K\subseteq N$. 
For each $\eta\in T$, we denote by $\nu_\eta\colon N\to N^{(T)}$ the canonical embedding from $N$ onto the $\eta$th direct summand of $N^{(T)}$.
Finally, let us denote by $G$ the submodule of $N^{(T)}$ generated by the set \[J = \{\nu_\eta(x)-\nu_\zeta(x)\mid \eta,\zeta\in T, (\exists \alpha<\kappa)(x\in K_\alpha\;\&\; \eta\restriction\alpha = \zeta\restriction\alpha)\}.\]

Then $N^{(T)}/G$ is called the \emph{tree module built from $N, \mathfrak K$ and $T$}. Moreover, for each $x\in K$, we denote by $\rho(x)$ the \emph{rank of $x$}, i.e.\ the least $\alpha<\kappa$ such that $x\in K_\alpha$. Note that $\rho(x)$ has to be a~successor ordinal unless $x = 0$.
\end{defn}

Notice that $G$ is actually the subgroup of $N^{(T)}$ generated by $J$: indeed, for $r\in R$, $(\nu_\eta(x)-\nu_\zeta(x))r = \nu_\eta(xr)-\nu_\zeta(xr)\in J$ whenever $\nu_\eta(x)-\nu_\zeta(x)\in J$ since $\rho(xr)\leq \rho(x)$. We will occasionally need to work with the elements in $G$. The following lemma will come in handy.

\begin{lem} \label{l:aux} In the setting of Definition~\ref{d:treemod}, let \[x = \sum_{i=1}^n (\nu_{\eta_i}(x_i)-\nu_{\zeta_i}(x_i)) \eqno{(+)}\] where the summands on the right-hand side are nonzero elements from $J$.

Assume that we have used the least amount of pairwise distinct elements from $T$ possible in the expression~$(+)$. Then, for each $\eta\in T$ occuring as a~lower index in $(+)$, the $\eta$th coordinate of $x$ in $N^{(T)}$ is nonzero.
\end{lem}

\begin{proof} Suppose that there is an $\eta\in T$ occuring in $(+)$ for which the $\eta$th coordinate of $x$ in $N^{(T)}$ is zero. Let us denote by $I\subseteq\{1,\dots,n\}$ the set of all indices such that either $\eta_i = \eta$, or $\zeta_i = \eta$; possibly replacing $x_i$ by $-x_i$, we can assume w.l.o.g.\ that it is always the case that $\eta_i = \eta$. Pick a~$k\in I$ with the maximal $\rho(x_k)$. Then $\eta\restriction\rho(x_k) = \zeta_k\restriction\rho(x_k)$, and so we can replace all the occurences of $\eta$ in $(+)$ with~$\zeta_k$. The summands then remain valid generators from $J$ (but notice that the $k$th one becomes zero). Moreover, the resulting sum remains the same, i.e.\ it equals~$x$. This is a~contradiction with our assumption on the minimality of $(+)$.
\end{proof}

The next lemma sums up some properties of tree modules that will be important in the sequel.

\begin{lem} \label{l:tree} Let $\kappa$ be an infinite regular cardinal, $N\in\ModR$ and $\mathcal A$ a~class of modules. Assume that $K$ is a~submodule of $N$ with a~filtration $\mathfrak K = (K_\alpha\mid \alpha\leq\kappa)$ such that $N/K_\alpha\in\mathcal A$ holds for each $\alpha<\kappa$. Put $C = N/K$. For each cardinal $\xi$, there exists a~$\lambda\geq\xi$ and a~short exact sequence $0\to D\to L\overset{g}{\to} C^{(2^\lambda)}\to 0$ such that
\begin{enumerate}
	\item $L$ is the direct limit of a~$\kappa^+$-continuous directed system consisting of modules from $\Filt(\mathcal A)$, and
	\item $|D|\leq\lambda$.
\end{enumerate}
\end{lem}

\begin{proof} Let $f\colon N\to C$ be the canonical projection. Take a cardinal $\lambda$ such that $|K|+\xi\leq\lambda = \lambda^{<\kappa}$ and $\lambda^\kappa = 2^\lambda$. Let $T = {}^\kappa \lambda$; then, trivially, $|T| = \lambda^\kappa = 2^\lambda$.

Put $N^\prime = N^{(T)}$. For each $\eta\in T$, we denote by $\nu_\eta\colon N\to N^\prime$ the canonical embedding from $N$ onto the $\eta$th direct summand of $N^\prime$. As in Definition~\ref{d:treemod}, let $G$ be the submodule of $N^\prime$ generated by the set \[J = \{\nu_\eta(x)-\nu_\zeta(x)\mid \eta,\zeta\in T, (\exists \alpha<\kappa)(x\in K_\alpha\;\&\; \eta\restriction\alpha = \zeta\restriction\alpha)\}.\]
We denote by $\pi\colon N^\prime\to N^\prime/G$ the canonical projection and, for each $S\subseteq T$, put $L_S = N^{(S)}+G/G$ and $L = L_T$. So $L$ is the tree module built from $N, \mathfrak K$ and $T$.

\smallskip

To verify that $(1)$ holds true, it suffices to show that $L_S\in\Filt(\mathcal A)$ for each $S\subseteq T$ of cardinality at most $\kappa$. Indeed, $L$ is the directed union of its submodules $L_S$ where $S$ runs through subsets of $T$ of cardinality at most $\kappa$. The associated directed system is clearly $\kappa^+$-continuous.

So let $S = \{\eta_\alpha\mid \alpha<\theta\}$ where $\theta\leq \kappa$ is a~cardinal. Furthermore, for each $\alpha\leq\theta$, set $S_\alpha = \{\eta_\beta\mid \beta<\alpha\}$. We claim that $(L_\alpha\mid \alpha\leq\theta)$ is an $\mathcal A$-filtration of $L_S$ where we put, for convenience, $L_\alpha:= L_{S_\alpha}$ for each $\alpha\leq\theta$.

First of all, $L_0 = 0$ and $L_\alpha = \bigcup_{\beta<\alpha}L_\beta$ for each limit $\alpha\leq\theta$. Take any $\alpha<\theta\leq\kappa$. Observe that \[L_{\alpha+1}/L_\alpha \cong \frac{N^{(S_{\alpha+1})}+G}{N^{(S_\alpha)}+G} \cong \frac{N^{\{\eta_\alpha\}}}{(N^{(S_\alpha)}+G)\cap N^{\{\eta_\alpha\}}}.\] Since $\kappa$ is regular and $|S_\alpha|<\kappa$, we have
\[\gamma := \sup\{\beta<\kappa\mid (\exists \zeta\in S_\alpha)\,\eta_\alpha\restriction \beta = \zeta\restriction\beta\}<\kappa.\]
We claim that $(N^{(S_\alpha)}+G)\cap N^{\{\eta_\alpha\}} = K_\gamma^{\{\eta_\alpha\}}$. First, the inclusion `$\supseteq$' is clear by the definition of $\gamma$ and $G$. (If $\gamma$ does not belong to the set whose supremum $\gamma$ is, we have to use that $K_\gamma = \bigcup_{\delta<\gamma} K_\delta$.) Let us show the other inclusion.

Take any $z = y + x\in N^{\{\eta_\alpha\}}$ where $y\in N^{(S_\alpha)}$, \[x = \sum_{i=1}^n (\nu_{\eta_i}(x_i)-\nu_{\zeta_i}(x_i)) \eqno{(+)}\] belongs to $G$ and the sum comprises of (nonzero) elements from $J$. Assume that, in this sum, we have used the least amount of pairwise distinct elements of $T$ possible. Since $x = z-y\in N^{(S_{\alpha+1})}$, we know by Lemma~\ref{l:aux} that the generators appearing in $(+)$ use only elements from $S_{\alpha+1}$. Using the definition of $\gamma$, we immediately deduce that $z\in K_\gamma^{\{\eta_\alpha\}}$. The second inclusion is proved.

\smallskip

It follows that \[L_{\alpha+1}/L_\alpha \cong \frac{N^{\{\eta_\alpha\}}}{(N^{(S_\alpha)}+G)\cap N^{\{\eta_\alpha\}}} \cong N/K_\gamma\]
which belongs to $\mathcal A$ by our assumption. Thus $L_S\in\Filt(\mathcal A)$.

\smallskip

For (2), we first pick for each $u\in{}^{<\kappa}\lambda = \{h\colon \alpha\to \lambda\mid \alpha<\kappa\}$ an $\eta_u\in T$ which extends~$u$. Since $\lambda^{<\kappa} = \lambda\geq |K|$, the cardinality of the submodule $F = \sum_{u\in\,{}^{<\kappa}\lambda}\nu_{\eta_u}(K)$ of $N^\prime$ is at most $\lambda$. By the definitions of $F$ and $G$, the canonical projection $N^\prime\to N^\prime/(F+G)$ coincides with the coproduct map $f^{(T)}\colon N^\prime \to C^{(T)}$. So $f^{(T)} = g\pi$ for the natural projection $g\colon L\to C^{(T)}$. Put $D = \Ker(g) = F + G/G$. Then $|D|\leq |F|\leq \lambda$.

\end{proof}

In the next lemma, we recall a~typical setting which provides us with a~module $N$ and a~filtration $\mathfrak K$ as in the hypothesis of Lemma~\ref{l:tree}. We call a~filtration $(A_\alpha\mid \alpha\leq \tau)$ \emph{pure} if $A_\alpha$ is pure in $A_\sigma$ (equivalently in $A_{\alpha+1}$) for all $\alpha<\sigma$.

\begin{lem} \label{l:dirlim} Let $\kappa$ be an infinite regular cardinal and $\mathcal N = (N_\alpha, g_{\alpha\beta}\colon N_\beta \to N_\alpha \mid \beta < \alpha<\kappa)$ be (a~well-ordered) $\kappa$-continuous directed system of modules. Assume that $N_0 = 0$. Put $C = \varinjlim N$ and $N = \bigoplus_{\alpha<\kappa} N_\alpha$. Then the canonical ($\kappa$-pure) presentation of $C$

\[\mathcal E\colon 0\longrightarrow K\overset{\subseteq}{\longrightarrow} N \overset{f}{\longrightarrow} C \longrightarrow 0\]
is the direct limit of the $\kappa$-continuous well-ordered directed system of pure short exact sequences
\[\mathcal E_\alpha\colon 0\longrightarrow K_\alpha\overset{\subseteq}{\longrightarrow} \bigoplus_{\beta<\alpha}N_\beta \overset{f_\alpha}{\longrightarrow} N_{\alpha^\prime} \longrightarrow 0,\quad \alpha<\kappa,\] where we set $\alpha^\prime = \alpha$ if $\alpha$ is limit (including zero), and $\alpha^\prime = \alpha-1$ if $\alpha$ is a~successor ordinal. Furthermore, $f_\alpha\restriction N_\beta = g_{\alpha^\prime\beta}$ for each $\beta<\alpha$. Finally, the connecting morphisms between $\mathcal E_\beta$ and $\mathcal E_\alpha$, for $\beta<\alpha$, comprise (from left to right) of inclusion, canonical split inclusion and $g_{\alpha^\prime\beta^\prime}$. (We put $g_{\gamma\gamma} = \ident_{N_\gamma}$ for each $\gamma<\kappa$.)

In particular, setting $K_\kappa = K$, we get a~pure filtration $\mathfrak K= (K_\alpha\mid \alpha\leq\kappa)$ of $K$.
\end{lem}

\begin{proof} Easy. Notice that $\mathcal E_\alpha$ is a~split short exact sequence if $\alpha<\kappa$ is a~successor ordinal or zero. As a~consequence, $\mathcal E_\alpha$ is pure for each $\alpha<\kappa$. In particular, $\mathfrak K$ is pure as well. See also \cite[Lemma~2.1]{GG0}.
\end{proof}

Now, we are ready to prove the promised push-down result.

\begin{thm} \label{t:push-down} Let $\kappa$ be an infinite regular cardinal and $\mathcal A$ be a~class of modules closed under taking direct limits of $\kappa^+$-continuous directed systems. Then $\mathcal A$ is closed under taking direct limits of $\kappa$-continuous directed systems provided that either one of the following conditions hold:
\begin{enumerate}
	\item[(i)] $\mathcal A = {}^\perp\mathcal B$ for a~class $\mathcal B\subseteq\ModR$;
	\item[(ii)] $\mathcal A = \Filt(\mathcal S)$ for a~set $\mathcal S\subseteq\ModR$ and $\mathcal A$ is closed under direct summands.
\end{enumerate}
\end{thm}

\begin{proof} We prove both cases, (i) and (ii), together. As the first step, we are going to reduce the problem to inspecting the direct limits of \emph{well-ordered} $\kappa$-continuous directed systems \emph{indexed by $\kappa$}: suppose that $\mathcal M = (M_i,h_{ji}\colon M_i\to M_j \mid i\leq j\in I)$ is an \emph{arbitrary} $\kappa$-continuous directed system of modules from $\mathcal A$. We have to show that $\varinjlim\mathcal M \in\mathcal A$. For the sake of nontriviality, assume that $(I,\leq)$ does not have the largest element. In particular, $|I|\geq\kappa$.

Consider the directed system, $\mathcal C$, formed by the direct limits of directed subsystems of $\mathcal M$ of cardinality $\kappa$ and canonical factorization maps between them. (Thus if $g\colon C_1\to C_2$ is a~homomorphism in $\mathcal C$, then $C_1 = \varinjlim \mathcal D_1$, $C_2 = \varinjlim \mathcal D_2$ where $\mathcal D_1, \mathcal D_2$ are directed subsystems of $\mathcal M$ of cardinality $\kappa$ such that $\mathcal D_1\subset\mathcal D_2$.) It follows that $\varinjlim\mathcal M = \varinjlim\mathcal C$ and $\mathcal C$ is $\kappa^+$-continuous. We claim that $\mathcal C$ consists of modules from $\mathcal A$; this would immediately result in $\varinjlim\mathcal M = \varinjlim\mathcal C\in\mathcal A$ using our assumption on $\mathcal A$.

\smallskip

So let $C\in\mathcal C$ be arbitrary. Since $\mathcal M$ is $\kappa$-continuous, the $C$ is, in fact, the direct limit of a~$\kappa$-continuous well-ordered subsystem $(N_\alpha, g_{\alpha\beta}\colon N_\beta\to N_\alpha \mid \beta<\alpha<\kappa)$ of $\mathcal M$. We thus reduced our considerations to the promised special type of directed systems.

Let us assume, w.l.o.g., that $N_0 = 0$. Put $N = \bigoplus_{\alpha<\kappa} N_\alpha$, and consider the filtration $\mathfrak K = (K_\alpha\mid \alpha\leq\kappa)$ from Lemma~\ref{l:dirlim}. Then $N/K_\alpha\cong \bigoplus_{\alpha\leq\beta<\kappa} N_\beta \oplus \bigoplus_{\beta<\alpha} N_\beta/K_\alpha \in\mathcal A$ for each $\alpha<\kappa$ since $\bigoplus_{\beta<\alpha} N_\beta/K_\alpha \cong \Img(f_\alpha) = N_{\alpha^\prime}\in\mathcal A$. Here, $\alpha^\prime$ is as in Lemma~\ref{l:dirlim}.

We fix an arbitrary $B\in\mathcal B$ and put $\xi = |B|$ if we use case (i). If in case (ii), let $\xi$ be such that all modules in $\mathcal S$ are $\xi$-presented. From Lemma~\ref{l:tree}, we obtain a~short exact sequence \[0\longrightarrow D\longrightarrow L\overset{g}{\longrightarrow} C^{(2^\lambda)}\longrightarrow 0\] for a~cardinal $\lambda\geq \xi$. By part (1) of the lemma and our assumption, it follows that $L\in\mathcal A$. By part (2), we know that $|D|\leq\lambda$.

If (i) holds true, we conclude everything with the classic Hunter's cardinal counting argument: applying the functor $\Hom_R(-,B)$ on the short exact sequence above, we get
\[0\longrightarrow\Hom_R(C,B)^{2^\lambda}\longrightarrow\Hom_R(L,B)\longrightarrow\Hom_R(D,B)\overset{\delta}{\longrightarrow}\Ext^1_R(C,B)^{2^\lambda}\longrightarrow 0.\]
From $|D|,|B|\leq\lambda$, it follows that $|\Hom_R(D,B)|\leq 2^\lambda$, and so $\Ext^1_R(C,B) = 0$ since $\delta$ is a~surjective mapping.
For an arbitrary $B\in\mathcal B$, we managed to show $C\in {}^\perp B$, whence $C\in {}^\perp\mathcal B = \mathcal A$.

\smallskip

If (ii) holds true, the class $\mathcal A$ is $\lambda$-Kaplansky by \cite[Theorem~10.3(a)]{GT0}, and so we can find a~$\lambda$-presented submodule $P$ of $L$ such that $D\subseteq P$ and $C^{(2^\lambda)}/g(P)\cong L/P\in \mathcal A$. However, $g(P)$ is $\lambda$-generated whence there is an~$X\subset 2^\lambda$ with $|X| \leq \lambda$ such that \[C^{(2^\lambda)}/g(P)\cong \frac{C^{(X)}}{g(P)}\oplus C^{(2^\lambda\setminus X)} \in \mathcal A.\] Since $\mathcal A$ is closed under direct summands, we finally get that $C\in\mathcal A$.
\end{proof}

\begin{exm} \label{e:ML} In Theorem~\ref{t:push-down}, it is not enough to assume that $\mathcal A = \Filt(\mathcal A)$ and that $\mathcal A$ be closed under direct summands and under taking direct limits of $\kappa^+$-continuous directed systems. As an example, take any ring $R$ which is not right perfect and let $\mathcal A$ be the class of all $\aleph_1$-projective (equivalently, flat Mittag-Leffler) modules, cf.\ \cite[Section~3.2]{GT0}.

By \cite[Corollary~3.20(a)]{GT0}, $\mathcal A = \Filt(\mathcal A)$ and $\mathcal A$ is even closed under pure submodules. Moreover, $\mathcal A$ is closed under taking direct limits of $\aleph_1$-continuous directed systems by \cite[Proposition~2.2]{HT0}. Nevertheless, $\mathcal A$ is not closed under countable direct limits since, over rings which are not right perfect, there exist countably presented flat modules which are not ($\aleph_1$-)projective by the well-known result due to Hyman Bass.
\end{exm}

Before stating the last corollary of this section, let us recall that a~ring $R$ is \emph{right $\kappa$-Noetherian} if every right ideal of $R$ is $\kappa$-generated.

\begin{cor} \label{c:aleph_n} Let $n<\omega$ and let $(\mathcal A,\mathcal B)$ be a~cotorsion pair generated by a~class of $<\!\aleph_n$-presented modules. More generally, let $\mathcal A = \Filt(\mathcal S)$ where $\mathcal S$ is a~class of $<\!\aleph_n$-presented modules. Then $\mathcal A$ is a~covering class if and only if $\mathcal A$ is closed under direct limits.

In particular, for every $m,n<\omega$, Enochs Conjecture holds over a~right $\aleph_n$-noetherian ring for the class $\mathcal P_m$ consisting of all modules of projective dimension at most $m$.
\end{cor}

\begin{proof} There is a representative set of $<\!\aleph_n$-presented modules, whence $\mathcal A$ is precovering by \cite[Theorem~6.11]{GT0} (or, more generally, by \cite[Theorem~7.21]{GT0}), and the if part follows from the well-known Enochs' result, cf.\ \cite[Theorem~5.31]{GT0}. The only-if part follows from Theorem~\ref{t:pure-epi} and $n$-fold application of Theorem~\ref{t:push-down}; since $\mathcal A$ is covering, it is closed under direct summands. Notice that any directed system of modules is $\aleph_0$-continuous by definition.

The last statement follows by \cite[Lemma~8.9]{GT0}.
\end{proof}

\section{Consistency of the Enochs Conjecture for $\Filt(\mathcal S)$ where $\mathcal S$ is a~set}
\label{sec:cons}

We shall work with the following additional set-theoretic hypothesis consistent with ZFC. Recall, that a~set $E\subseteq\kappa^+$ is \emph{non-reflecting} if $E\cap \delta$ is non-stationary for each ordinal $\delta<\kappa^+$ with uncountable cofinality.

\medskip

\noindent\textbf{Assumption $(*)$.} \textit{For each infinite regular cardinal $\theta$, there is a~proper class of cardinals $\kappa$ such that:}
\begin{enumerate}
	\item \textit{There exists a~non-reflecting stationary set $E\subseteq\kappa^+$ consisting of ordinals with cofinality $\theta$, and}
 	\item \textit{$\kappa^{<\theta} = \kappa$.}
\end{enumerate}

\smallskip

The assumption $(*)$ holds, for instance, in the constructible universe L. More generally, $(*)$ holds true if the square principle $\square_\kappa$ holds for each strong limit singular cardinal $\kappa$. On the other hand, $(*)$ cannot hold, e.g., in a~model of ZFC which contains a~strongly compact cardinal.

\smallskip

We are going to recall a~result from \cite[Proof of 2.1]{ES} saying that, for $\kappa$ as in Assumption~$(*)$, there exist almost free trees of size $\kappa^+$ and height $\theta$ in the sense of the following definition, cf.\ \cite[1.3 and 1.5]{ES}.

\begin{defn} \label{d:trees} Let $T\subseteq {}^\theta \lambda$ where $\lambda, \theta$ are infinite regular cardinals.
\begin{enumerate}
	\item For a~cardinal $\xi$, we say that $T$ is \emph{$\xi$-free} if, for each $S\subseteq T$ with $|S|<\xi$, there exists a~map $\psi\colon S\to \theta$ such that $\{\{\eta\restriction \gamma\mid \psi(\eta)<\gamma<\theta\}\mid \eta\in S\}$ is a~system of pairwise disjoint sets. Moreover, for technical reasons, we require that $\psi(\eta)$ is either 0 or a~successor ordinal for each $\eta\in S$.
	\item We call $T$ \emph{nonfree} if there is a~one-one enumeration $T = \{\eta_\alpha\mid \alpha<\lambda\}$ and a~stationary set $F\subseteq \lambda$ such that, for each $\delta\in F$ and every ordinal $\mu<\theta$, \[\eta_\delta\restriction\mu\in\{\eta_\alpha\restriction\mu \mid \alpha<\delta\}.\]
\end{enumerate}
If $T$ is $\lambda$-free and nonfree, then $T$ is called an \emph{almost free tree of size $\lambda$ and height~$\theta$}.
\end{defn}

Technically, the elements of the set $T$ in the definition above encode\footnote{An $\eta\in T$ encodes the branch $\{\eta\restriction\mu\mid \mu<\theta\}$.} branches of the tree $\{\eta\restriction\mu \mid \eta \in T, \mu<\theta\}$ which is ordered by mapping extension, i.e.\ by inclusion if we formally treat mappings as sets of ordered pairs.

Vaguely put, $\xi$-freeness entails that any $<\!\xi$ branches have nonempty, pairwise disjoint end segments. The nonfreeness, on the other hand, means that all the branches can be well-ordered in a~way that it happens very often (i.e.\ on a~stationary set) that each initial segment of a~branch coincides with the corresponding initial segment of some smaller branch in the well ordering.

For the reader's convenience, we include here, with a~slightly supplemented proof, the result on the existence of almost free trees which appears on page 503 in \cite{ES}.

\begin{prop} \label{p:trees} Let $\kappa$ be an infinite cardinal and put $\lambda = \kappa^+$. Assume that $\kappa^{<\theta} = \kappa$ and that there exists a~non-reflecting stationary set $E\subseteq\lambda$ consisting of ordinals with cofinality $\theta$. Then there exists an~almost free $T\subseteq {}^\theta\lambda$.
\end{prop}

\begin{proof} For each $\delta\in E$ fix a~strictly increasing mapping $\zeta_\delta\colon \theta\to\lambda$ whose range forms a~cofinal subset of $\delta$. Set $T = \{\zeta_\delta\mid \delta\in E\}$. First, we show that $T$ is $\lambda$-free.

Let $S\subset E$ be arbitrary with $|S|<\lambda$. By induction on the least $\sigma<\lambda$ such that $S\subseteq \sigma$, we show the existence of $\psi\colon S\to \theta$ such that $\{\{\zeta_\alpha\restriction \gamma\mid \psi(\alpha)<\gamma<\theta\}\mid \alpha\in S\}$ is a~system of pairwise disjoint sets. If $\sigma = 0$, we are trivially done. There are two possible scenarios otherwise.

If $\sigma = \delta + 1$ for $\delta\in S$, we put $\psi(\delta) = 0$. From the inductive assumption, we get $\psi^\prime\colon S\setminus\{\delta\}\to\theta$. For $\alpha\in S\setminus\{\delta\}$, we then simply define \[\psi(\alpha) = \max\{\psi^\prime(\alpha), \min\{\mu<\theta \mid \zeta_\delta\restriction \mu\neq \zeta_\alpha\restriction\mu\}\}.\]

If $\sigma$ is a limit ordinal, we use that $E$ is non-reflective to fix in $\sigma$ a~closed and unbounded set $C$ such that $C\cap E = \varnothing$ and $0\in C$. For each $\beta\in C$, we denote by $\beta^*$ the successor of $\beta$ in $C$ and put $S_\beta = \{\alpha\in S\mid \beta<\alpha<\beta^*\}$. We see that $S$ is the (disjoint) union of all the sets $S_\beta$. Using the inductive hypothesis, we obtain $\psi_\beta\colon S_\beta\to \theta$ for each $\beta\in C$. We can assume, without loss of generality, that $\zeta_\alpha(\psi_\beta(\alpha))>\beta$ for each $\alpha\in S_\beta$ and every $\beta\in C$: indeed, the range of $\zeta_\alpha$ is cofinal in $\alpha$ and $\beta<\alpha$. Our desired $\psi\colon S\to \theta$ can be now defined as the union $\bigcup_{\beta\in C}\psi_\beta$. For two distinct $\beta,\beta^\prime\in C$, $\alpha\in S_\beta, \alpha^\prime\in S_{\beta^\prime}$ and $\gamma>\psi(\alpha), \gamma^\prime>\psi(\alpha^\prime)$, it cannot happen that $\zeta_\alpha\restriction \gamma = \zeta_{\alpha^\prime}\restriction \gamma^\prime$ since $\beta<\zeta_\alpha(\psi(\alpha))<\beta^*$ whilst $\beta^\prime<\zeta_{\alpha^\prime}(\psi(\alpha^\prime))<{\beta^\prime}^*$.

\smallskip

Now, we show that $T$ is nonfree.  Let $C$ denote the set consisting of all $\gamma<\lambda$ such that, for each $\alpha\in E$ and $\mu<\theta$ where (the range of) $\zeta_\alpha\restriction\mu$ is bounded below~$\gamma$, there exists $\beta<\gamma, \beta\in E$ such that $\zeta_\alpha\restriction \mu = \zeta_\beta\restriction\mu$. Since $\kappa^{<\theta} = \kappa$, the set $C$ is closed and unbounded in $\lambda$: indeed, for any $\gamma_0<\lambda$, there is at most $\kappa$ mappings from $\mu$ into $\gamma$ for any $\mu<\theta$; as a~consequence, there exists $\gamma_1\geq\gamma_0$ satisfying that, for all $\alpha\in E$ such that $\zeta_\alpha\restriction\mu$ is bounded below $\gamma_0$ for some $\mu<\theta$, there exists $\beta<\gamma_1, \beta\in E$ such that $\zeta_\alpha\restriction\mu = \zeta_\beta\restriction\mu$; by the same reasoning, we find $\gamma_2\geq\gamma_1$ and so on, and finally, we put $\gamma = \sup\{\gamma_n\mid n<\omega\}$; then $\gamma\in C$ (using the bounded below property), and $C$ is thus unbounded; finally, utilizing the bounded below property once more, it immediately follows that $C$ is closed.

Since $E$ is stationary in $\lambda$, the set $C\cap E$ is stationary as well. Let $f\colon \lambda\to E$ denote the order-preserving isomorphism and, for each $\alpha<\lambda$, put $\eta_\alpha = \zeta_{f(\alpha)}$. Define $F$ as the full preimage of $C\cap E$ in $f$. Pick any $\delta\in F$ and set $\gamma = f(\delta)$. Clearly $\gamma\in C\cap E$ and $\eta_\delta = \zeta_\gamma\colon \theta \to \gamma$. Since $\zeta_\gamma$ is strictly increasing, it follows that, for each $\mu<\theta$, $\zeta_\gamma\restriction\mu$ is bounded below $\gamma$ whence $\zeta_\gamma\restriction\mu\in\{\zeta_\beta\restriction\mu\mid\beta<\gamma, \beta\in E\}$ by the definition of $C$. The latter amounts to $\eta_\delta\restriction\mu\in\{\eta_\alpha\restriction\mu\mid\alpha<\delta\}$. We have proved that $T$ is nonfree.
\end{proof}

The following result can be of independent interest. In a~forthcoming paper by the second-named author and Manuel Cort\'es-Izurdiaga, it will be used to show that, under a~slightly stronger hypothesis than Assumption~$(*)$, $\kappa$-pure injectivity amounts to pure injectivity over any ring $R$ and for any (uncountable) cardinal $\kappa$.

Given $\mathcal A\subseteq\ModR$, we say that $\mathcal A$ is \emph{closed under pure filtrations} if $A_\tau\in\mathcal A$ holds for each pure $\mathcal A$-filtration $(A_\alpha\mid\alpha\leq\tau)$. In particular, a~class closed under pure filtrations is closed under direct sums.

\begin{prop} \label{p:corecon} Let $\theta$ be a~regular cardinal, and let $\mathcal A\subseteq\ModR$ be closed under pure filtrations and direct summands. Suppose that a~$\theta$-continuous well-ordered directed system $(N_\alpha,g_{\alpha\beta}\colon N_\beta\to N_\alpha \mid \beta<\alpha<\theta)$ consisting of modules from $\mathcal A$ is given. Assume that there exists $\kappa\geq |R|+\sum_{\alpha<\theta}|N_\alpha|$ satisfying $(1)$ and $(2)$ from Assumption~$(*)$ for the given $\theta$. Put $\lambda = \kappa^+$.

Then there is a~stationary subset $F\subseteq \lambda$ and an~$L\in\ModR$ possessing a~filtration $\mathfrak L = (L_\alpha\mid \alpha\leq\lambda)$ satisfying
\begin{enumerate}
	\item[(a)] $L_\alpha\in\mathcal A$ and $|L_\alpha|<\lambda$ for each $\alpha<\lambda$, and
	\item[(b)] for every $\delta\in F$ and $\delta<\varepsilon<\lambda$, we have $\varinjlim_{\alpha<\theta}N_\alpha\cong L_{\delta+1}/L_\delta$, and the latter module splits in $L_{\varepsilon}/L_\delta$.
\end{enumerate}
\end{prop}

\begin{proof} We can assume, w.l.o.g., that $N_0 = 0$. As in the proof of Theorem~\ref{t:push-down}, put $N = \bigoplus_{\alpha<\theta} N_\alpha$, and consider the ($\theta$-pure) presentation of $C = \varinjlim _{\alpha<\theta} N_\alpha$
\[\mathcal E\colon 0\longrightarrow K\overset{\subseteq}{\longrightarrow} N \overset{f}{\longrightarrow} C \longrightarrow 0.\]

Lemma~\ref{l:dirlim} provides us with a~pure filtration $\mathfrak K = (K_\alpha \mid \alpha\leq\theta)$ of $K$. Furthermore, using the notation from the lemma, $N/K_\alpha\cong \bigoplus_{\alpha\leq\beta<\theta} N_\beta \oplus \bigoplus_{\beta<\alpha} N_\beta/K_\alpha \in\mathcal A$ for each $\alpha<\theta$ since $\bigoplus_{\beta<\alpha} N_\beta/K_\alpha \cong \Img(f_\alpha) = N_{\alpha^\prime}\in\mathcal A$.

Moreover, $\mathfrak K$ is an~$\mathcal A$-filtration: indeed, $K_\alpha$ is a~direct summand in $\bigoplus_{\beta<\alpha}N_\beta$ which is a~canonical direct summand in $N\in\mathcal A$ if $\alpha$ is a~successor ordinal or $\alpha = 0$, and so $K_\alpha$ splits in $K_{\alpha+1}\in\mathcal A$ in this case; if $\alpha>0$ is limit, then we have $K_{\alpha+1}/K_\alpha\cong N_\alpha\in\mathcal A$ since the rightmost connecting map from $\mathcal E_\alpha$ to $\mathcal E_{\alpha+1}$ is $\ident_{N_\alpha}$.

\smallskip

By the property of $\kappa$, we have $\kappa\geq |N|$ and $\kappa^{<\theta} = \kappa$. Moreover, there is a~non-reflecting stationary set $E\subseteq\lambda$ consisting of ordinals of cofinality $\theta$. Let $T\subseteq {}^\theta\lambda$ be an~almost free tree provided by Proposition~\ref{p:trees}.

As in Definition~\ref{d:treemod}, consider the module $N^\prime = N^{(T)}$. For each $\eta\in T$, denote by $\nu_\eta\colon N\to N^\prime$ the canonical embedding from $N$ onto the $\eta$th direct summand of~$N^\prime$. Let $G$ be the submodule of $N^\prime$ generated by the set \[J = \{\nu_\eta(x)-\nu_\zeta(x)\mid \eta,\zeta\in T, (\exists \alpha<\theta)(x\in K_\alpha\;\&\; \eta\restriction\alpha = \zeta\restriction\alpha)\}.\]
Furthermore, denote by $\pi\colon N^\prime\to N^\prime/G$ the canonical projection and, for each $S\subseteq T$, put $L_S = N^{(S)}+G/G$ and $L = L_T$.

Pick any $S\subseteq T$ with $|S|<\lambda$ and fix a~map $\psi\colon S\to\theta$ from the definition of $\lambda$-freeness. We claim that $L_S\in\mathcal A$.

First, we fix, for each $\alpha$ in the range of $\psi$, a~submodule $P_\alpha$ of $N$ such that $K_\alpha\oplus P_\alpha = N$; remember that $\alpha$ is zero or a~successor ordinal by our technical requirement on $\psi$. Then put $V = \bigoplus_{\eta\in S} P_{\psi(\eta)} \subseteq N^{(S)}$, $Y_1 = V+G/G \subseteq L_S$ and $Y_0 = 0$. We claim that $V\cap G = 0$ from which it would readily follow that $Y_1\cong V\in \mathcal A$.

Assume, for the sake of contradiction, that there exists a~nonzero $x\in V\cap G$ and express it, using nonzero generators from $J$, as \[x = \sum_{i=1}^n (\nu_{\eta_i}(x_i)-\nu_{\zeta_i}(x_i)). \eqno{(+)}\] As in the statement of Lemma~\ref{l:aux}, we can assume that $(+)$ uses the least number of pairwise distinct elements from $T$ possible. Choose $\eta\in T$ with the maximal $\psi(\eta)$ among all the elements from $T$ appearing in $(+)$. Let $I$ denote the set of all the $i\in\{1,\dots,n\}$ such that the $i$th summand in $(+)$ contains $\eta$ as a~lower index; we can assume, w.l.o.g., that $\eta = \eta_i$ in such a~case. We claim that $x_i\in K_{\psi(\eta)}$ for every $i\in I$. Suppose not.

Using the notion of a~rank from Definition~\ref{d:treemod}, we see that $\rho(x_i)>\psi(\eta)\geq\psi(\zeta_i)$ and, since the $i$th summand in $(+)$ is an element of $J$, we get $\eta\restriction\rho(x_i) = \zeta_i\restriction\rho(x_i)$ which is in contradiction with the property of $\psi$. But now, since $x\in V$ and $x_i\in K_{\psi(\eta)}$ for $i\in I$, we have $\sum_{i\in I} x_i = 0$, i.e.\ the $\eta$th coordinate of $x$ in $N^\prime$ is zero; a~contradiction with Lemma~\ref{l:aux}.

\smallskip

We are going to build a~pure $\mathcal A$-filtration $\mathfrak Y = (Y_\beta\mid \beta\leq \sigma)$ of $L_S$, thus proving that $L_S\in\mathcal A$. For each $\beta>0$, the $Y_\beta$ will have the form $(N^{(Q_\beta)}+V+G)/G$ where $Q_\beta\subseteq S$, and $Q_\beta \subsetneq Q_\gamma$ if $\beta<\gamma\leq\sigma$. Our $Y_1$ has got the correct form, and we have no other choice in the limit steps than to take the union of the preceding ones; otherwise put, if $Y_\beta$ is defined for all $\beta<\gamma$ and $\gamma$ is limit, we set $Q_\gamma = \bigcup_{\beta<\gamma} Q_\beta$. If $\beta<\lambda$ is such that $Y_\beta = L_S$, we put $\sigma = \beta$ and we are done. (We must end up with $\sigma<\lambda$ since $|S|<\lambda$ and the subsets $Q_\beta$ of $S$ form a~strictly increasing chain.)

Otherwise, $Y_\beta\subsetneq L_S$ and the set $U = \{\eta\in S\mid L_{\{\eta\}}\nsubseteq Y_\beta\}$ is nonempty. Set $Q = S\setminus U$ and notice that $Q_\beta\subseteq Q$ and $Y_\beta = (N^{(Q)}+V+G)/G$. Choose an $\eta\in U$ with the minimal $\psi(\eta)$ possible among all the elements in $U$. Put $\delta = \psi(\eta)$ and $\gamma = \sup\{\alpha<\theta\mid (\exists \zeta\in Q)\,\eta\restriction\alpha = \zeta\restriction\alpha\}$. Then $\pi\nu_\eta(K_\gamma)\subseteq Y_\beta$ which implies $\gamma<\delta$ since $\pi\nu_\eta(P_{\delta})\subseteq Y_\beta$ and $\eta\in U$. We claim that $\nu_\eta(K_{\delta})\cap (N^{(Q)}+V+G) = \nu_\eta(K_\gamma)$. We have already settled the inclusion `$\supseteq$'. Let us prove the other one.

Pick an arbitrary $x+y+z\in \nu_\eta(K_{\delta})$ where $y\in N^{(Q)}$, $z\in V$ and \[x = \sum_{i=1}^n (\nu_{\eta_i}(x_i)-\nu_{\zeta_i}(x_i)), \eqno{(+)}\] where the summands in $(+)$ are nonzero and belong to $J$. As before, we assume that the expression of $x$ in $(+)$ uses the least amount of pairwise distinct elements from $T$ possible. Lemma~\ref{l:aux} implies that all the elements from $T$ occuring in $(+)$ belong to $S$. Notice also that no $\zeta\in U\setminus\{\eta\}$ occurs in $(+)$: indeed, if a~$\zeta\in U\setminus\{\eta\}$ occurs in $(+)$, we choose one with the maximal $\psi(\zeta)$ among all of them. According to Lemma~\ref{l:aux}, the $\zeta$th coordinate of $x$ in $N^\prime$ is nonzero. On the other hand, the $\zeta$th coordinate of $x+y+z$ is zero since $x+y+z\in \nu_\eta(K_{\delta})$. Moreover, $y$ does not contribute to this coordinate, and so the $\zeta$th coordinate of $x$ belongs to $\nu_\zeta(P_{\psi(\zeta)})$. In particular, there has to be an~$i\in\{1,\dots, n\}$ such that $\zeta$ occurs in the $i$th summand of $(+)$, w.l.o.g.\ $\zeta = \zeta_i$, and $\rho(x_i)>\psi(\zeta)$. But then $\eta_i\restriction \rho(x_i) = \zeta\restriction \rho(x_i)$ which is absurd: if $\eta_i\in Q$, it would entail that $\Img(\nu_\zeta)\subseteq N^{(Q)}+V+G$ and $\zeta\notin U$; and if $\eta_i\in U$, then $\psi(\zeta)\geq\psi(\eta_i)$ holds by the choices of $\zeta$ and $\eta$ (if $\eta_i = \eta$), and we get a~contradiction with the property of $\psi$.

We have demonstrated that only the elements from $Q\cup\{\eta\}$ occur as lower indices in~$(+)$. By the definition of $\gamma$, the $\eta$th coordinate of $x$ belongs to $K_\gamma$, and the $\eta$th coordinate of $z$ (and, trivially, $y$) is zero since $x+y+z\in \nu_\eta(K_{\delta})$. We have proved that $\nu_\eta(K_{\delta})\cap (N^{(Q)}+V+G) = \nu_\eta(K_\gamma)$.

Finally, we can set $Q_{\beta+1} = Q\cup\{\eta\}$ and $Y_{\beta+1} = (N^{(Q_{\beta+1})}+V+G)/G$. Then \[Y_{\beta+1}/Y_\beta \cong \frac{N^{(Q_{\beta+1})}+V+G}{N^{(Q)}+V+G} = \frac {\nu_\eta(K_{\delta})+(N^{(Q)}+V+G)}{N^{(Q)}+V+G} \cong \frac{\nu_\eta(K_{\delta})}{\nu_\eta(K_\gamma)}\cong \frac{K_\delta}{K_\gamma}\in \mathcal A.\] Also notice that $Y_\beta$ is pure in $Y_{\beta+1}$ as witnessed by the commutative diagram below\footnote{The vertical arrows in the diagram are embeddings since $\Img(\nu_\zeta) \cap G = 0$ holds for each $\zeta\in T$.}
\[\xymatrix{0 \ar[r] & Y_\beta \ar[r]^-{\subseteq} & Y_{\beta+1} \ar[r] & K_\delta/K_\gamma \ar[r] & 0 \\
0 \ar[r] & K_\gamma \ar[r]_-{\subseteq} \ar[u]^-{\pi\nu_\eta\restriction K_\gamma} & K_\delta \ar[u]^-{\pi\nu_\eta\restriction K_\delta} \ar[r] & K_\delta/K_\gamma \ar@{=}[u] \ar[r] & 0
}\] and the fact that $\mathfrak K$ is pure. This completes the construction of the~pure $\mathcal A$-filtration $\mathfrak Y$ of $L_S$. We have shown that $L_S$, indeed, belongs to $\mathcal A$.

\smallskip

Using that $T$ is nonfree, we can fix an~enumeration $T = \{\eta_\alpha\mid \alpha<\lambda\}$ and a~stationary $F\subseteq \lambda$ as in Definition~\ref{d:trees}. Notice that $|L_S|<\lambda$ for $|S|<\lambda$, by our choice of $\kappa$. Let $\mathfrak L$ denote the filtration $(L_{S_\alpha} \mid \alpha\leq\lambda)$ where $S_\alpha = \{\eta_\beta \mid \beta<\alpha\}$ for each $\alpha\leq\lambda$. We know that (a) holds for $\mathfrak L$.

To verify (b), choose a~$\delta\in F$ and an~$\varepsilon<\lambda$ with $\delta<\varepsilon$. We have to show that $C\cong L_{S_{\delta +1}}/L_{S_\delta}$ and that the latter module is a~direct summand in $L_{S_\varepsilon}/L_{S_\delta}$.

From the property of the stationary set $F$, we get that $\eta_\delta\restriction\mu\in\{\eta_\alpha\restriction\mu \mid \alpha<\delta\}$ holds for each ordinal $\mu<\theta$. It yields $\pi\nu_{\eta_\delta}(K) = L_{S_\delta} \cap L_{\{\eta_\delta\}}$; the inclusion `$\supseteq$' follows from the fact that the generators in $J$ use only elements $x\in K$. In particular, $L_{S_{\delta +1}}/L_{S_\delta} \cong L_{\{\eta_\delta\}}/\pi\nu_{\eta_\delta}(K) \cong N/K \cong C$. Arguing once more by the form of elements in $J$, we observe that $L_{\{\eta_\delta\}}\cap L_{S_\varepsilon\setminus\{\eta_\delta\}} \subseteq \pi\nu_{\eta_\delta}(K) \subseteq L_{S_\delta}$, whence \[\frac{L_{S_{\delta +1}}}{L_{S_\delta}} \oplus \frac{L_{S_\varepsilon\setminus\{\eta_\delta\}}}{L_{S_\delta}} = L_{S_\varepsilon}/L_{S_\delta}.\]
\end{proof}

The ``tree module'' construction in the proof of Proposition~\ref{p:corecon} was suggested already in \cite[Hypothesis~3.3a, 1.4(bis) and 1.5(bis)]{ES}. Our nonfree part, i.e.\ the last three paragraphs of the proof, corresponds to 1.5(bis) and works pretty much as suggested. However, the proof of the $\lambda$-free part could not be carried out along the lines of 1.4(bis) mostly because of a~flaw (albeit eventually fixable) in the proof of \cite[Lemma~1.4]{ES}: the basis $Y_\tau$ therein is constructed without any regard to the bases $B_\alpha$ of $N_\alpha$, $\alpha<\tau$, which, in general, leads to a~clash between $Y_\tau$ and some $B_\alpha$, making the last sentence of the proof false.

\smallskip

The main result of this section follows. By Corollary~\ref{c:Silvana}, it applies, in particular, to the left-hand classes $\mathcal A$ of cotorsion pairs generated by a~set of modules in case $\A$ is covering. In conjunction with Corollary~\ref{c:Silvana}, it gives a~consistently positive answer to \cite[Open problem~5.4.1]{GT0}.

\begin{thm} \label{t:maincon} Assume $(*)$. Let $\mathcal A\subseteq \ModR$ and $\xi$ be an uncountable regular cardinal such that $\mathcal A = \Filt(\mathcal A^{<\xi})$ and $\mathcal A$ is closed under direct summands and taking direct limits of $\xi$-continuous well-ordered directed systems of modules (and monomorphisms between them). Then $\mathcal A$ is closed under direct limits.
\end{thm}

\begin{proof} It is enough to show that, for each infinite regular cardinal $\theta$, the class $\mathcal A$ is closed under taking direct limits of $\theta$-continuous well-ordered directed systems of length $\theta$. Let $(N_\alpha,g_{\alpha\beta}\colon N_\beta\to N_\alpha \mid \beta<\alpha<\theta)$ be any such system; so $N_\alpha\in\mathcal A$ for each $\alpha<\theta$.

Using Assumption~$(*)$, we obtain a~suitable $\kappa\geq |R|+\xi+\sum_{\alpha<\theta}|N_\alpha|$ for our~$\theta$. Set $\lambda = \kappa^+$. Let $F\subseteq \lambda$ be a~stationary set and $\mathfrak L = (L_\alpha\mid \alpha\leq\lambda)$ be a~filtration provided by Proposition~\ref{p:corecon}. In particular, $L = L_\lambda$ is the directed union of $\lambda$-continuous well-ordered directed system consisting of the submodules $L_\alpha\in\mathcal A$, $\alpha<\lambda$. Thus we get $L\in\mathcal A$ by our assumption.

Since $\lambda>\xi+|R|$ and $\mathcal A \subseteq \Filt(\mathcal A^{<\xi})$, there is an $\mathcal A$-filtration $(M_\alpha\mid \alpha\leq\lambda)$ of~$L$ with $|M_\alpha|<\lambda$ for each $\alpha<\lambda$. It follows that the set $X = \{\alpha<\lambda \mid L_\alpha = M_\alpha\}$ is closed and unbounded. Choose a~$\delta\in X\cap F$ and an $\varepsilon\in X$ with $\varepsilon > \delta$. Then, by (b) from Proposition~\ref{p:corecon}, the direct limit $\varinjlim_{\alpha<\theta} N_\alpha$ is isomorphic to a~direct summand of $L_\varepsilon/L_\delta = M_\varepsilon/M_\delta\in\mathcal A$. It follows that $\varinjlim_{\alpha<\theta} N_\alpha\in\mathcal A$.
\end{proof}

\begin{rema} \label{r:unaware} The assumption $\mathcal A = \Filt(\mathcal A^{<\xi})$ in the statement of Theorem~\ref{t:maincon} can be relaxed to `$\mathcal A \subseteq \Filt(\mathcal A^{<\xi})$ and $\mathcal A$ is closed under pure filtrations'. This slight generalisation covers also the classes $\mathcal A$ which are not necessarily closed under extensions, e.g.\ definable classes of modules (which, of course, are trivially closed under direct limits). Apart from some classes of the form $\Add(M)$, treated in detail in \cite{Sa2}, the authors of this paper are unaware of any example of a~covering class~$\mathcal A$ of modules not satifying this relaxed assumption for a~suitable cardinal $\xi$.
\end{rema}

\begin{cor} \label{c:conEnochs} Assume $(*)$. Let $\mathcal S \subseteq \ModR$ be a~set. Put $\mathcal A = \Filt(\mathcal S)$. Then the following conditions are equivalent:
\begin{enumerate}
	\item $\mathcal A$ is a~covering class of modules;
	\item $\mathcal A$ is closed under direct limits;
	\item $\mathcal A$ is closed under direct summands and under taking direct limits of $\xi$-continuous well-ordered directed systems of modules (and monomorphisms between them) for an~infinite regular cardinal $\xi$.
\end{enumerate}
\end{cor}

\begin{proof} Theorem~\ref{t:maincon} covers the implication $(3) \Longrightarrow (2)$; the only one using Assumption $(*)$. Note that we can suppose w.l.o.g.\ that each module in $\mathcal S$ is $<\!\xi$-presented.

The implication $(1) \Longrightarrow (3)$ yields from Corollary~\ref{c:Silvana}. Finally, $(2) \Longrightarrow (1)$: by \cite[Theorem~7.21]{GT0}, $\mathcal A$ is a~precovering class and the rest follows by the classic Enochs' result.
\end{proof}

\begin{rema} \label{r:aware}
Theorem~\ref{t:maincon} and the implications $(3) \Longrightarrow (1)$ and $(3) \Longrightarrow (2)$ in Corollary~\ref{c:conEnochs} cannot be proved in ZFC alone (unless the existence of strongly compact cardinals is refutable in ZFC which is considered to be highly implausible).

In any model of ZFC with a~strongly compact cardinal $\xi$, the class $\mathcal A$ of free abelian groups, and---more generally---of projective modules over any ring $R$ with $|R|<\xi$, is closed under $\xi$-pure-epimorphic images by \cite[Theorem~3.3]{ST}. In particular, $\mathcal A$ satisfies the hypotheses of Theorem~\ref{t:maincon} and $(3)$ in Corollary~\ref{c:conEnochs}. However, $\mathcal A$ is closed under $\varinjlim$ (equivalently, $\mathcal A$ is covering) only if the ring $R$ in question is right perfect.
\end{rema}

Regardless of the remark above, it is still possible that Enochs Conjecture can be proved in ZFC.

\bigskip

\noindent\emph{Acknowledgement}: The authors very much appreciate the comments and suggestions made by the anonymous referee.


\end{document}